\documentclass[12pt]{amsart}

\newtheorem{theorem}{Theorem}[section]

\newtheorem{proposition}[theorem]{Proposition}
\newtheorem{corollary}[theorem]{Corollary}
\theoremstyle{definition}
\newtheorem{definition}[theorem]{Definition}

\theoremstyle{remark}

\numberwithin{equation}{section}
\def\id{{\bf 1}\!\!{\rm I}}
\def\g{\gamma}
\def\G{\Gamma}
\def\w{\omega}
\def\ch{{\mathcal H}}
\def\m{\mu}
\def\t{\tau}
\begin{document}

\title[Conditionally expectations and martingales ]{Conditionally expectations and martingales
in noncommutative $L_p$-spaces associated with center-valued
traces}
\author{Inomjon Ganiev}

\address{Department of Science in Engineering,
Faculty of Engineering, International Islamic University Malaysia,
P.O. Box 10, 50728 Kuala-Lumpur, Malaysia}
\email{inam@iium.edu.my}

\author{Farrukh Mukhamedov}
\address{Farrukh Mukhamedov\\
 Department of Computational \& Theoretical Sciences\\
Faculty of Science, International Islamic University Malaysia\\
P.O. Box, 141, 25710, Kuantan\\
Pahang, Malaysia} \email{{\tt far75m@yandex.ru} {\tt
farrukh\_m@iium.edu.my}}

\begin{abstract}
In this paper we prove the existence of conditional expectations in
the noncommutative $L_p(M,\Phi)$ spaces associated with
center-valued traces. Moreover, their description is also provided.
As an application of the obtained results, we establish the norm
convergence of the martingales in noncommutative $L_p(M,\Phi)$
spaces. \vskip 0.3cm \noindent


\end{abstract}

\maketitle

\section{Introduction}

 One of important nations of probability theory is a
conditionally expectation operators, which plays general roll
theory martingales, ergodic theory and et al.  Existence and basic
properties of conditional expectation operators provided in many
books of functional analysis and the theory of probability.

To the existence of conditional expectations on operator algebras
was first studied by Umegaki's works \cite{U1,U2}. Furthermore, the
existence criterion was provided in \cite{Ta}. Then the martingale
convergence theory for conditional expectations in von Neumann
algebras has been developed (see for example, \cite{DN,La,Ts}). One
of the important arias in the theory of operator algebras, is the
integration theory for traces and weights defined on von Neumann
algebras (se for example \cite{S,N,Y}). In this directions,
analogous of $L_p$-spaces have been investigated. In \cite{CP,G,HT}
the martingale convergence in non-commutative $L_p$-spaces was
obtained.

On the other hand, development of the theory of integration for
measures $\mu$ with the values in Dedekind complete Riesz spaces
has inspired the study of (bo)-complete lattice-normed
$L_p$-spaces (see, for example, \cite{K2}). The existence of
center-valued traces on finite von Neumann algebras naturally
leads to develop the theory of integration for this kind of
traces. In \cite{GaC} non-commutative $L_p$-spaces assiciated with
with central-valued traces have been investigated.  In \cite{CK}
an abstract characterization of these spaces have been provided.
Going further, in \cite{CZ} more general $L_p$-spaces associated
with Maharam traces have been studied.

Therefore, main aim of this paper to investigate the existence of
conditional expectations in this $L_p$-spaces. Note that first
vector-valued analog existence conditional expectations was given in
\cite{Kus} for commutative von Neumann algebras. It is known
\cite{GaC} that $L_p$-spaces associated with central-valued traces
are Banach-Kantorovich spaces. The theory of  Banach-Kantorovich
spaces is now sufficiently well-developed (for instance, see
\cite{Kus,K2}). The methods of Boolean valued analysis play an
important role in the theory. Using these methods, the
Banach-Kantorovich lattices and the corresponding homomorphisms can
be interpreted as the Banach lattice and bounded linear operators
within a suitable Boolean valued model of set theory \cite{Kus}.
This approach to the theory of  Banach-Kantorovich lattices makes ut
possible to use the transfer principle \cite{Kus} for obtaining
various properties of  Banach-Kantorovich lattices that are
analogous of the corresponding properties of classical Banach
lattices. Naturally, when using such a method, an additional study
is needed for establishing the required interrelation between the
objects of 2-valued and Boolean valued models of set theory.

Another important approach to study  Banach-Kantorovich spaces is
provided by the theory of continuous and measurable Banach bundles
\cite{G0}-\cite{G2}. In this approach the representation of a
Banach-Kantorovich lattice as a space of measurable sections of a
measurable Banach bundle makes it possible to obtain the needed
properties of the lattice by means of the corresponding stalkwise
verification of the properties. As an application of this approach,
in \cite{Ga} it has been given a representation of the conditional
expectations as measurable bundle of classical conditional
expectations, and proved martingale convergence theorems. Moreover,
noncommutative $L_p(M,\Phi)$-spaces associated with center-valued
traces are represented as bundle of noncommutative $L_p$-spaces
associated with numerical traces \cite{GaC}. In \cite{CK} an
abstract characterization of these spaces are provided. Certain
other extensions of these $L_p$-spaces have been investigated in
\cite{CZ}. For other applications of the mentioned method, we refer
the reader to \cite{AAK,CGa},\cite{GM}-\cite{GM4}.

In the present paper, we mainly employ the last mentioned approach
to establish the existence of conditional expectations on
noncommutative $L_p(M,\Phi)$-spaces associated with center-valued
traces. Moreover, we describe such expectations and as an
application of the obtained results, we prove the norm convergence
of martingales in noncommutative $L_p(M,\Phi)$-spaces. We note that
the obtained results open new perspective in the field of martingale
convergence.

\section{Preliminaries}

In this section we recall necessary notions and facts which will
be used in the next sections.

Let $(\Omega,\Sigma,\mu)$ be a measurable space with finite
measure $\mu$, and $L_0(\Omega)$ be the algebra of all measurable
functions on $\Omega$ (as usual, a.e. equal functions are
identified). By ${\mathcal{L}}^{\infty}(\Omega)$ we denote the set
of all measurable essentially bounded functions on $\Omega$, and
$L^{\infty}(\Omega)$ denote an algebra of equivalence classes of
essentially bounded measurable functions.

Let $\mathcal{U}$ be a linear space over the real field
$\mathbb{R}$.  By $\|\cdot\|$ we denote a $L_0(\Omega)$-valued
norm on $U$. Then the pair $(\mathcal{U},\|\cdot\|)$ is called a
{\it lattice-normed space (LNS) over $L_0(\Omega)$}. An LNS
$\mathcal{U}$ is said to be {\it $d$-decomposable} if for every
$x\in \mathcal{U}$ and the decomposition $\|x\|=f+g$ with $f$ and
$g$ disjoint positive elements in $L_0(\Omega)$ there exist
$y,z\in \mathcal{U}$ such that $x=y+z$ with $\|y\|=f$, $\|z\|=g$.

Suppose that $(\mathcal{U},\|\cdot\|)$ is an LNS over
$L_0(\Omega)$. A net $\{x_\alpha\}$ of elements of $\mathcal{U}$
is said to be {\it $(bo)$-converging} to $x\in \mathcal{U}$ (in
this case, we write $x=(bo)$-$\lim x_\alpha$), if the net
$\{\|x_\alpha - x\|\}$ $(o)$-converges to zero in $L_0(\Omega)$
(written as $(o)$-$\lim \|x_\alpha -x\|=0$). A net
$\{x_\alpha\}_{\alpha\in A}$ is called {\it $(bo)$-fundamental} if
$(x_\alpha-x_\beta)_{(\alpha,\beta)\in A\times A}$
$(bo)$-converges to zero.

An LNS in which every $(bo)$-fundamental net $(bo)$-converges is
called {\it $(bo)$-complete}. A {\it Banach-Kantorovich space
(BKS) over $L_0(\Omega)$} is a $(bo)$-complete $d$-decomposable
LNS over $L_0(\Omega)$. It is well known \cite{G1,G2} that every
BKS $\mathcal{U}$ over $L_0(\Omega)$ admits an
$L_0(\Omega)$-module structure such that $\|fx\|=|f|\cdot\|x\|$
for every $x\in \mathcal{U},\ f\in L_0(\Omega)$, where $|f|$ is
the modulus of a function $f\in L_0(\Omega)$.

A set  $B \subset \mathcal{U}$ is called \emph{bounded}, if set
$\{\|x\|: x \in B\}$ is order bounded in $ L_{0}$.
 An operator $T: \mathcal{U} \rightarrow \mathcal{U}$  is called
  $L_{0}$-linear, if $T(\alpha x + \beta y)= \alpha T(x) + \beta T(y)$
  for all $\alpha,\beta \in L_{0}$ and $x,y \in \mathcal{U}$.
 An $L_{0}$-linear operator $T$ is called $ L_{0}$--bounded,
  if for any bounded set  $B$ in $\mathcal{U}$,
   the set  $T (B)$ bounded in $\mathcal{U}$.
  For an $L_{0}$-bounded operator  $T$ we put
 $$\|T\|=\sup\{\|T(x)\|: \|x\|\leq \bf{1}\},$$  where $\bf{1}$ - identity element in  $L_{0}$

Let $H$ be a Hilbert space, let $B(H)$ be the $*$-algebra of all
bounded linear operators on $H,$ and let $I$ be the identity
operator on $H.$  Given  a von Neumann algebra $M$ acting on $H,$
denote by $Z(M)$  the center of $M$ and by $P(M)$  the lattice of
all projections in $M$. Let  $P_{fin}(M)$ be the set of all finite
projections in $M.$  A densely-defined  closed linear operator $x$
(possibly unbounded) affiliated with $M$ is said to be
\emph{measurable} if there exists a sequence
$\{p_n\}_{n=1}^{\infty}\subset P(M)$ such that $p_n\uparrow
\mathbf{1}$, \ $p_n(H)\subset \mathfrak{D}(x)$ and
$p_n^\bot=\mathbf{1}-p_n \in P_{fin}(M) $ for every $n=1,2,\ldots$
(here $\mathfrak{D}(x)$ is the domain of $x$). Let us denote by
$S(M)$ the set  of all measurable  operators. If $M$ is a
commutative von Neumann algebra, it is $*$-isomorphic to the
$*$-algebra $L_\infty(\Omega,\Sigma,\mu)$ of all essentially
bounded complex measurable functions  with the identification
almost everywhere, where $(\Omega,\Sigma,\mu)$ is a measurable
space.  In addition $S(M)\cong L_0(\Omega,\Sigma,\mu)$ \cite{S}.

Let $x,y$ be measurable  operators. Then $x+y,~xy$ and $x^*$ are
densely-defined and preclosed. Moreover, the closures
$\overline{x+y}$ (strong sum), $\overline{xy}$ (strong product)
and $x^*$ are also measurable, and $S(M)$ is a  $*$-algebra with
respect to the strong sum, strong product, and the adjoint
operation (see \cite{S}). For any subset $E\subset S(M)$ we denote
by $E_{sa}$ (resp. $ E_+$ ) the set of all self-adjoint (resp.
positive ) operators from $E.$

For  $x\in S(M)$  let  $x=u|x|$ be the polar decomposition, where
$|x|=(x^*x)^{\frac{1}{2}},$ $u$ is a partial isometry in $B(H).$
Then $u\in M$ and $|x|\in S(M).$ If $x\in S_h(M)$ and
$\{E_\lambda(x)\}$ are the spectral  projections of $x,$ then
$\{E_\lambda(x)\}\subset P(M).$

The locally measure topology  $t(M)$ on $L_0(\Omega,\Sigma,\mu)$
is by definition the linear (Hausdorff) topology whose fundamental
system of neighborhoods of $0$ is given by $$
W(B,\varepsilon,\delta)=\{f\in\ L_0(\Omega,\, \Sigma,\, \mu)
\colon \hbox{ there exists a set } \ E\in \Sigma, \mbox{ such
that} $$
$$  \ E\subseteq B, \ \mu(B\setminus E)\leq\delta, \
f\chi_E \in L^\infty(\Omega,\Sigma,\mu), \
\|f\chi_E\|_{{L_\infty}(\Omega,\Sigma,\mu)}\leq\varepsilon\}.
$$ Here \ $\varepsilon, \ \delta $ run over all strictly positive
numbers and  $B\in\Sigma$, \ $\mu(B)<\infty.$ It is known that
$(S(M),t(M))$ is a complete topological $*$-algebra. Note that a
net $\{f_\alpha\}$ converges  locally in measure to $f$ (notation:
$f_\alpha \stackrel{t(M)}{\longrightarrow}f$) if and only if
$f_\alpha \chi_B $ converges in $\mu$-measure to $f\chi_B$ for
each $B\in\Sigma$ with $\mu(B)<\infty$.

Let $M$ be any finite von Neumann algebra, $S(M)$ be the set all
measurable operators affiliated to $M$. Let $Z$ be some subalgebra
of the center $Z(M)$. Then one may identify $Z$ with $*$-algebra
$L_\infty(\Omega,\Sigma,m)$ and do $S(Z)$ with
$L_0(\Omega,\Sigma,m)$. Recall that \textit{a center valued (i.e.
$Z$-valued) trace} on the von Neumann algebra $M$ is a $Z$-linear
mapping $\Phi:M\to Z$ with $\Phi(x^*x)=\Phi(xx^*)\geq 0$ for all
$x\in M$. It is clear that $\Phi(M_+)\subset Z_+$. A trace $\Phi$
is said to be \textit{faithful} if the equality $\Phi(x^*x)=0$
implies $x=0$, \textit{normal} if
$\Phi(x_{\alpha})\uparrow\Phi(x)$ for every $x_{\alpha},x\in
M_{sa}$, $x_{\alpha}\uparrow x$. Note that the existence of such
kind of traces has been studied in \cite{CZ1}.

Let now $M$ be an arbitrary finite von Neumann algebra, $\Phi$ be
a center-valued trace on $M$. The locally measure topology $t(M)$
on $S(M)$ is  the linear (Hausdorff) topology whose fundamental
system of neighborhoods of $0$ is given by $$ V(B,\varepsilon,
\delta ) = \{x\in S(M)\colon \ \mbox{there exists } \ p\in P(M),
z\in P(Z(M)) $$  $$ \mbox{ such that} \ xp\in M,
\|xp\|_{M}\leq\varepsilon, \ z^\bot \in W(B,\varepsilon,\delta), \
\Phi_M(zp^\bot)\leq\varepsilon z\},$$ where $\|\cdot\|_{M}$ is the
$C^*$-norm in $M.$  It is known that $(S(M),t(M))$ is a complete
topological $*$-algebra \cite{Y}.

From \cite[\S 3.5]{Mur_m} we have the following criterion for
convergence  in the topology $t(M). $

\begin{proposition}\label{2.1.} A net $\{x_\alpha\}_{\alpha\in A} \subset S(M)$ converges to zero in the topology $t(M)$   if and only if $\Phi_M(E^\bot_\lambda (|x_\alpha|) \stackrel{t(M)}{\longrightarrow} 0$ for any $\lambda>0.$
\end{proposition}

Following \cite{CZ} an operator  $x\in S(M)$ is said to be {\em
$\Phi$-integrable} if there exists a sequence $\{x_n\} \subset M$
such that $x_n \stackrel{t(M)}{\to} x $ and $\|x_n-x_m\|_\Phi
\stackrel{t(Z)}{\longrightarrow} 0$ as $n,m \to \infty.$

Let $x$ be a $\Phi$-integrable operator from $ S(M).$ Then  there
exists a $\widehat{\Phi}(x)\in S(Z)$ such that $\Phi(x_n)
\stackrel{t(Z)}{\longrightarrow}\widehat{\Phi}(x).$ In addition
$\widehat{\Phi}(x)$ does not depend on the choice of a sequence
$\{x_n\}\subset M,$ for which
$x_n\stackrel{t(M)}{\longrightarrow}x,$ $\Phi(|x_n-x_m|)
\stackrel{t(Z)}{\longrightarrow}0$ \cite{CZ}. It is clear that
each operator $x\in M$ is $\Phi$-integrable and $\widehat{\Phi}
(x)=\Phi(x).$

Denote by $L_1(M,\Phi)$ the set of all $\Phi$-integrable operators
from $S(M).$ If $x\in S(M)$ then $x\in L_1(M,\Phi)$ iff $|x|\in
L_1(M,\Phi),$  in addition $|\widehat{\Phi}(x)|\leq
\widehat{\Phi}(|x|)$ \cite{CZ1}. For any $x\in L_1(M,\Phi),$ set
$\|x\|_{1,\Phi}=\widehat{\Phi}(|x|).$ It is known that
$L_1(M,\Phi)$ is a linear subspace of $S(M),$ $ML_1(M,\Phi)M
\subset L_1(M,\Phi),$ and $x^*\in L_1(M,\Phi)$ for all $x\in
L_1(M,\Phi)$ \cite{CZ1}. Let $\Phi$ be a faithful normal trace on
$M$ with values in $L_0(\Omega)$ (associated with  $Z$). Following
\cite{CZ} for any $p > 1$, we set $L_p(M,\Phi) = \{x \in S(M) :
\|x\|_{p,\Phi} \in L_1(M,\Phi)\}$ and $\|x\|_{p,\Phi} =
\widehat{\Phi}(|x|^p)^{\frac{1}{p}}$.

\begin{proposition}\cite{CZ1,CZ} \label{1.1} Let $x,y\in S(M)$. Then the
following statements hold:
\begin{enumerate}
\item[(i)]$ \|x\|_{1,\Phi}=\sup\{|\Phi(xy)|: \|y\|_M\leq 1\}.$

\item[(ii)] for $p>1$, one has $\|x\|_{p,\Phi}=\sup\{|\Phi(xy)|:
\|y\|_{q,\Phi}\leq \mathbf{1}\},$ where
$\frac{1}{p}+\frac{1}{q}=1.$
\end{enumerate}
\end{proposition}
%
%\begin{proof} (i). As $|\Phi(xy)|\leq \Phi(|xy|)\leq \|x\|_1\|y\|_M$ we
%get $\sup\{|\Phi(xy)|: \|y\|_M\leq 1\}\leq\|x\|_1.$ Let $x=u|x|$
%be polar decomposition of $x$ and $y=u^*.$ Then
%$\Phi(xy)=\Phi(xu^*)=\Phi(u^*x)=\Phi(|x|)=\|x\|_1.$ Hence
%$\|x\|_1=\sup\{|\Phi(xy)|: \|y\|_M\leq 1\}.$
%
%(ii) proved in \cite{CZ} (see Theorem 4.3).
%
%\end{proof}

In the early 90's of the last century A.E. Gutman \cite{G1}
introduced  the
 measurable Banach bundles with lifting axioms. It was established that
every Banach--Kantorovich space over the ring of measurable
functions can be expressed as measurable bundle of Banach spaces.
Now let us recall some notions from this approach.

Let $X$ be a mapping which maps every point $\omega\in \Omega$
    to some Banach space  $(X(\omega),\|\cdot\|_{X(\omega)})$.  In what follows,
    we assume that $X(\omega)\neq \{0\}$  for all
    $\omega\in \Omega.$
A function $u$
  is said to be a \textit{section} of $X$,
   if it is defined almost everywhere in  $\Omega$
    and takes its value $u(\omega)\in X(\omega)$
      for   $\omega\in dom(u),$
      where  $\omega\in dom(u)$ is the domain of
      $u.$ Let   $L$ be some set of sections.

\begin{definition} \cite{G1}. A pair  $(X, L)$ is said to be
a {\it measurable bundle
 of Banach spaces} over $\Omega$   if
\begin{enumerate}
\item[1.]  $\lambda_1 c_1+\lambda_2 c_2\in L$
  for all  $\lambda_1, \lambda_2\in \mathbb{R}$
   and $c_1, c_2\in L,$ where
   $\lambda_1 c_1+\lambda_2 c_2:\omega\in dom(c_1)\cap
   dom(c_2) \rightarrow \lambda_1 c_1(\omega)+\lambda_2 c_2(\omega);$

\item[2.]  the function $||c||:\omega\in  dom(c)\rightarrow
||c(\omega)||_{X(\omega)}$
  is measurable for all $c\in L;$

\item[3.]  for every $\omega\in \Omega$
  the set  $\{c(\omega): c\in L, \omega\in dom(c)\}$
   is dense in $X(\omega).$
\end{enumerate}
   \end{definition}

   A section $s$ is a step-section, if there are pairwise disjoint sets
  $A_1,A_2,\ldots,A_n\in\Sigma$ and sections
 $c_1,c_2,\ldots,c_n\in L$ such that $\bigcup\limits_{i=1}^n
 A_i=\Omega$ è $s(\omega)=\sum\limits_{i=1}^n
 \chi_{A_i}(\omega)c_i(\omega)$ for almost all $\omega\in\Omega$.

 A section $u$ is measurable, if for any $A\in\Sigma$
 there is a sequence $s_n$ of step-sections such that
 $s_n(\omega)\rightarrow u(\omega)$ for almost all $\omega\in A$.

 Let $M(\Omega,X)$ be the set of all measurable sections. By symbol
 $L_0(\Omega,X)$  we denote factorization of the $M(\Omega,X)$ with respect to almost everywhere equality.
 Usually, by $\hat{u}$ we denote a class from $L_0(\Omega,X)$, containing a section $u\in M(\Omega,X)$,
 and by
 $\|\hat{u}\|$ we denote an element of $L_0(\Omega)$,
 containing $\|u(\omega)\|_{X(\omega)}$.
Let ${\mathcal L^{\infty}}(\Omega,X)=\{u\in
M(\Omega,X):\|u(\omega)\|_{X(\omega)}\in \mathcal
L^{\infty}(\Omega)\}$ and $L^{\infty}(\Omega,X)=\{\widehat{u}\in
L_0(\Omega,X): \|\widehat{u}\|\in L^{\infty}(\Omega)\}.$  One can
define the spaces $\mathcal {L^{\infty}}(\Omega,X)$ and
$L^{\infty}(\Omega,X)$ with real-valued norms $\|u\|_{\mathcal
L^{\infty}(\Omega,X)}=\sup\limits_{\omega\in
\Omega}|u(\omega)|_{X(\omega)}$ and
$\|\widehat{u}\|_{\infty}=\bigg\|\|\widehat{u}\|\bigg\|_{L^{\infty}(\Omega)},$
respectively.

 \begin{definition} Let $X,Y$ be measurable bundles of Banach spaces. A set linear operators $\{T(\omega) : X(\omega)\rightarrow
 Y(\omega)\}$ is called \textit{measurable bundle of linear operators} if
 $T(\omega)(u(\omega))$ is measurable section for any measurable
 section $u$.
\end{definition}

Let $(X,L)$ be a measurable bundle of Banach spaces. If each
$X(\omega)$ is a noncommutative $L_p$-space, i.e.
$X(\omega)=L_p(M(\omega),\tau_\omega)$, associated with finite von
Neumann algebras $M(\omega)$ and with strictly normal numerical
trace $\tau_\omega$ on
 $M(\omega)$, then  the measurable bundle  $(X,L)$ of Banach spaces is
 called \textit{ measurable bundle of noncommutative  $L_p$-spaces}.

 \begin{theorem} \label{1.2}\cite{GaC} There exists a measurable bundle $(X,L)$
 of noncommutative  $L_p$-spaces $L_p(M(\omega),\tau_\omega)$,
 such that  $L_0(\Omega,X)$ is Banach~---Kantorovich $*$-algebroid,
 which is isometrically and order  $*$-isomorph to  $L_p(M,\Phi)$. Moreover,  the isometric and
  order  $*$-isomorphism $H: L_p(M,\Phi)\rightarrow L_0(\Omega,X)$ can be chosen with the following propertires\\
  \begin{enumerate}
\item[(a)] $\Phi(x)(\omega) = \tau_\omega(H(x)(\omega))$ for all
 $x\in M$  and for almost all
 $\omega\in\Omega$;\\
 \item[(b)] $x\in M$ if and only if  $H(x)(\omega)\in
 M(\omega)$ a.e. and there exist positive number
 $\lambda>0$,
 that $\|H(x)(\omega)\|_{M(\omega)}\leq\lambda$ for almost all $\omega$;\\
 \item[(c)] $z\in Z$ if and only if
 $H(z)=(\widehat{z(\omega)\mathbf{1}_\omega)}$ for some
 $\widehat{z(\omega)}\in L_\infty(\Omega)$, where $\mathbf{1}_\omega$~---
 unit of algebra  $M(\omega)$, in particular
 $H(\mathbf{1})(\omega)=
 \mathbf{1}_\omega$ for almost all $\omega$.\\
 \item[(d)] the section $(H(x)(\omega))^*$ is measurable for all  $x\in L_p(M,\Phi)$.\\
 \item[(e)] the section $H(x)(\omega)\cdot H(y)(\omega)$ is measurable for all  $x,y\in M$.
 \end{enumerate}
 \end{theorem}

\section{The existence of vector-valued lifting}

In this section we establishes the existence of the lifting in a
non-commutative setting. Note that in the case of $C^*$-algebras,
the existence of the lifting has been given in \cite{GC1} (see also
\cite{GM3}).

Let $M$ be a von Neumann algebra. Then it can be identify with a
linear subspace of $L^\infty(\Omega,X)$ by the isomorphism $H$,
since if $x\in M$, then one has
$$\|H(x)\|_{L_0(\Omega,X)}=\|x\|_1=\Phi(|x|)\in L^\infty(\Omega)$$

\begin{theorem} \label{1.3} There exists a mapping  $\ell: M(\subset
L^\infty(\Omega,X))\rightarrow \mathcal{L}^\infty(\Omega,X)$ with
following properties
 \begin{enumerate}
\item[(a)]  for every $x\in M $ one has $\ell(x)\in x,\ {\rm dom}\ \ell (x)=\Omega$;\\
\item[(b)] if $x_1, x_2\in M$ and $\lambda_1, \lambda_2\in
\mathbb{R}$, then
 $\ell(\lambda_1 x_1+ \lambda_2 x_2)=\lambda_1\ell(x_1)+\lambda_2\ell(x_2)$;\\
\item[(c)]
 $\|\ell(x)(\omega)\|_{L_p(M(\omega),\tau_\omega)}=p(\|x\|_p)(\omega)$
 for all $x\in M$ and for all  $\omega\in\Omega$;\\
\item[(d)] if $x\in M,\ \lambda\in
 L^\infty(\Omega)$, then
 $\ell(ex)=p(e)\ell(x)$;\\
\item[(e)] if $x\in M$, then
 $\ell(x^*)=\ell(x)^*$;\\
\item[(f)] if $x,y\in M$, then
 $\ell(xy)=\ell(x)\ell(y)$;\\
 \item[(g)] the set $\{\ell(x)(\omega) : x\in M\}$ is dense
  in $L_p(M(\omega),\tau_\omega)$ for all
 $\omega\in\Omega$.
 \end{enumerate}
 \end{theorem}

\begin{proof} Following \cite{GaC} for every $x\in M$ we define
$$
\Phi_0(x)=\Phi(x)(1+\Phi(\id))^{-1}.
$$
One can see that $\Phi_0$ is an $L^{\infty}(\Omega)$-valued faithful
normal trace on $M$. By $\rho$ we denote the lifting on
$L^{\infty}(\Omega)$ (see \cite{G0}). Now define a finite trace
$\varphi_\omega$ on $M$ by
$\varphi_\omega(x)=\rho(\Phi_0(x))(\omega)$, where
$\omega\in\Omega$. Due to \cite[Lemma 6.4.1]{Dik} the function
$s_\omega(x,y)=\varphi_\omega(y^*x)$ is a bi-trace on $M$, and
therefore, the equality $\langle x,y\rangle_\omega=s_\omega(x,y)$
defines a quasi-inner product on $M$.

Denote $I_\omega=\{x\in M: \ s_\omega(x,x)=0\}$. It is known that
$I_\omega$ is a two-sided ideal in $M$, therefore, one considers the
quotient space $\Gamma_\omega=M/I_\omega$, by
$\pi_\omega:M\to\Gamma_\omega$ we denote the canonical mapping. The
involution and multiplication are defined on $\G_\w$ by the usual
way, i.e. $\pi_\w(x)^*=\pi_\w(x^*)$ and
$\pi_\w(x)\cdot\pi_\w(y)=\pi_\w(xy)$. According to \cite[Proposition
6.2.3]{Dik} $\G_\w$ is a Hilbert algebra. By $\ch(\w)$ we denote the
Hilbert space which is the completion of $\G_\w$; the inner product
in $\ch(\w)$ we denote by the same symbol, i.e.
$\langle\cdot,\cdot\rangle_\w$.

The mapping $\pi_\w(x)\to \pi_\w(y)\pi_\w(x)$, $x,y\in M$, can be
extended by continuity to a bounded linear operator $T_\w(y)$ on
$\ch(\w)$. It is known \cite{Dik} that $T_\w(x)$ is a representation
of $M$ in $\ch(\w)$. Let $M(\w)$ be the von Neumann algebra
generated by $T_\w(M)$, i.e. $M(\w)=T_\w(M)''$. By $\m_\w$ we denote
the natural trace on $M(\w)$ which is defined by
$\m_\w(\pi_\w(x))=\langle\pi_\w(x)\id_\w,\id_\w\rangle_\w$ for all
$\pi_\w(x)\in\G_\w$. One can see that $\m_\w$ is a faithful, normal
and finite trace on $M(\w)$ (see \cite[Proposition 6.8.3]{Dik}). Now
let us consider a non-commutative $L_1$-space $L_1(M(\w),\t_w)$,
where $\t_\w(\cdot)=(1+\Phi(\id))(\w)\m_\w(\cdot)$. By
$i_\w:\G_\w\to M(\w)$ one denotes the canonical embedding, and
$j_\w:(M(\w),\t_w)\to L_1(M(\w),\t_\w)$ denotes the natural
embedding. Then $\g_\w=j_\w\circ i_\w\circ\pi_\w$ is a linear
mapping  from $M$ to $L_1(M(\w),\t_\w)$.

Let us define
 $$\ell(x)(\omega)=\gamma_\omega(x)$$ for any $x\in M$.

(a) Since any element $x\in M$ is identified with the element
  $\widehat{\gamma_\omega(x)}$, then one has
 $\ell(x)\in x$ (see \cite{GaC}).

 (b) The linearity  of $\ell$ is obvious.

(c) Let $x\in M$. Then
\begin{eqnarray*}
\|\ell(x)(\omega)\|_{L_1(M(\omega),\tau_\omega)}&=&\|\gamma_\omega(x)\|_{L_1(M(\omega),\tau_\omega)}=\tau_\omega(|\pi_\omega(x)|)\\
&=&\tau_\omega(\pi_\omega(|x|))=\rho(\Phi(|x|))(\omega)\\
&=&\rho(\|x\|_1)(\omega)
\end{eqnarray*}
for all  $\omega\in\Omega$.

(d) Let $\chi_A\in L^\infty(\Omega)$ and $x\in M$, then
 $\chi_A\cdot x\in M$. By $\tilde{\Sigma}$ we denote a
 complete Boolean algebra of equivalent classes w.r.t. a.e. equality, of sets from
 $\Sigma$. The lifting $\rho :
 L^\infty(\Omega)\rightarrow\mathcal{ L}^\infty(\Omega)$ induces a lifting
   $\tilde{\rho} : \tilde{\Sigma}\rightarrow\Sigma$ such that $p(\chi_A)=\chi_{\tilde{p}(A)}$.
 Due to
 \begin{eqnarray*}
 \|\pi_\omega(\chi_A x)\|_{L_1(M(\omega),\tau_\omega)} &=&  \rho(\|\chi_A \cdot
 x\|_1)(\omega)=\rho(\chi_A)(\omega)\cdot
 p(\|x\|_1)(\omega)\\
 &=& \rho(\chi_A)(\omega)\cdot\|\pi_\omega(x)\|_{L_1(M(\omega),\tau_\omega)}\\
 &=&
 \chi_{\tilde{p}(A)}(\omega)\|\pi_\omega(x)\|_{L_1(M(\omega),\tau_\omega)},
 \end{eqnarray*} we obtain
 $\pi_\omega(\chi_A\cdot x)=0$, if $\omega\in\tilde{\rho}(A)$. Let
 $\omega\in\tilde{\rho}(A)$, then
\begin{eqnarray*}
 \|\pi_\omega(\chi_A\cdot
 x)-\pi_\omega(x)\|_{L_1(M(\omega),\tau_\omega)}&=& \|\pi_\omega(\chi_{\Omega\setminus
 A}\cdot x)\|_{L_1(M(\omega),\tau_\omega)}\\
 & =&\rho(\|\chi_{\Omega\setminus A}\cdot
 x\|_1)(\omega)= \rho(\chi_{\Omega\setminus A})(\omega)\cdot
 p(\|x\|_1)(\omega)\\
 &=& \chi_{\tilde{\rho}(\Omega\setminus
 A)}(\omega)\cdot \|\pi_\omega(x)\|_{L_1(M(\omega),\tau_\omega)}=0.
 \end{eqnarray*}  Therefore, we have $\pi_\omega(\chi_A \cdot x)= \chi_{\tilde{\rho}(A)}(\omega)
 \cdot\pi_\omega(x)= p(\chi_A)(\omega)\cdot\pi_\omega(x)$ for all
 $\omega\in\Omega$.

 Let $\lambda=\sum\limits_{i=1}^n r_i\chi_{A_i}\in
 L^\infty(\Omega)$ be a simple function. Then
 \begin{eqnarray*}\pi_\omega(\lambda
 x)&=& \pi_\omega\bigg(\sum\limits_{i=1}^n r_i \chi_{A_i} x\bigg)=
 \sum\limits_{i=1}^n \pi_\omega(r_i \chi_{A_i} x)\\
 &=&
 \sum\limits_{i=1}^n r_i \rho(\chi_{A_i})(\omega)\pi_\omega(x) =
\rho(\lambda)(\omega) \pi_\omega(u).
 \end{eqnarray*}

The density argument implies that for any $\lambda\in
L^\infty(\Omega)$ there exists a sequence of simple functions
  $\{\lambda_n\}$ such that
 $\|\lambda_n-\lambda\|_{L^\infty(\Omega)}\rightarrow0$ as $n\rightarrow\infty$.
 From
 \begin{eqnarray*}
 \|\pi_\omega(\lambda_n x)-\pi_\omega(\lambda
 x)\|_{L_1(M(\omega),\tau_\omega)}&=& \|\pi_\omega((\lambda_n - \lambda)x)\|_{L_1(M(\omega),\tau_\omega)} =
\rho(\|(\lambda_n-\lambda)x\|_1)(\omega)\\[2mm]
 &=& \rho(|\lambda_n -
 \lambda|)(\omega)\cdot \rho(\|x\|_1)(\omega)\\[2mm]
 &\leq&
 \|\rho(\lambda_n - \lambda)\|_{\mathcal{ L}^\infty(\Omega)}\cdot
 \|\rho(\|x\|_1)\|_{\mathcal{ L}^\infty(\Omega)}\\[2mm]
 &=&
 \|\lambda_n-\lambda\|_{L^\infty(\Omega)}\cdot \|\|x\|_1 \|_{L^\infty(\Omega)}
 \end{eqnarray*}
 one gets
 $\pi_\omega(\lambda
 x)=\lim\limits_{n\rightarrow\infty}\pi_\omega(\lambda_n x)$ for all
 $\omega\in\Omega$. So,
 $$\pi_\omega(\lambda
 x)=\lim\limits_{n\to\infty}\pi_\omega(\lambda_n x)=
 \lim\limits_{n\rightarrow\infty} \rho(\lambda_n)(\omega)\pi_\omega(x)=
\rho(\lambda)(\omega)\pi_\omega(u)$$
 for all $\omega\in\Omega$.

 Hence,
 $$i_\omega
 (\pi_\omega(\lambda x)=
 i_\omega(\rho(\lambda)(\omega)\pi_\omega(x))=
 \rho(\lambda)(\omega)i_\omega(\pi_\omega(x))= \rho(\lambda)(\omega)
 i_\omega (\pi_\omega(x))$$ and $j_\omega(\lambda x))=\rho(\lambda)(\omega)
 j_\omega (x)$ for all $\omega\in\Omega$. These mean that  $\ell(\lambda x)=\rho(\lambda)\ell(x)$
 for any $x\in M$ and $\lambda\in
 L^\infty(\Omega)$.

(e) According to  $\gamma_\omega(x^*)=\gamma_\omega(x)^*$ for any
$x\in M$ we get $\ell(x^*)=\ell(x)^*$.

 (f) From
$\gamma_\omega(xy)=\gamma_\omega(x)\gamma_\omega(y)$ for any $x,y\in
M$ it follows that  $\ell(xy)=\ell(x)\ell(y)$ for any $x,y\in M$.

(g) By the construction of $\gamma_\omega$ the set
$\{\gamma_\omega(x): x\in M\}$ is dense in
$L_1(M(\omega),\tau_\omega)$ for any $\omega\in\Omega$. Therefore,
the set $\{\ell(x)(\omega) : x\in M\}$ is dense
  in $L_1(M(\omega),\tau_\omega)$ for all
 $\omega\in\Omega$.

The proof is complete. \end{proof}

\begin{definition}
The defined map $\ell$ in Theorem \ref{1.3} is called \textit{a
noncommutative vector-valued lifting associated with the lifting
$\rho$}.
\end{definition}
%
% \begin{theorem} \label{1.4} \cite{Ga1} There exists a measurable bundle $(X,L)$
% of noncommutative  $L_p$-spaces $L_p(M(\omega),\tau_\omega)$ such
% that  $L^0(\Omega,X)$ is a Banach--Kantorovich $*$-algebroid,
% which us isometrically and order  $*$-isomorph  to $L_p(M,\Phi)$. In addition, an isometric and
%  order  $*$-isomorphism
% $H: L_p(M,\Phi)\rightarrow L^0(\Omega,X)$ can be chosen with the following properties
% \begin{enumerate}
%\item[(a)] $\Phi(x)(\omega) = \tau_\omega(H(x)(\omega))$ for all
% $x\in M$  and for almost all
% $\omega\in\Omega$.\\
%\item[(b)] $x\in M$ if and only if when  $H(x)(\omega)\in
% M(\omega)$ a.e. and there exist positive number
% $\lambda>0$,
% that $\|H(x)(\omega)\|_{M(\omega)}\leq\lambda$ for almost all $\omega$.\\
%\item[(c)] $z\in Z$ if and only if when
% $H(z)=(\widehat{z(\omega)\mathbf{1}_\omega)}$ for some
% $\widehat{z(\omega)}\in L_\infty(\Omega)$, where $\mathbf{1}_\omega$~---
% unit of algebra  $M(\omega)$, in particular
% $H(\mathbf{1})(\omega)=
% \mathbf{1}_\omega$ for almost all $\omega$.\\
%\item[(d)]Section $(H(x)(\omega))^*$ is measurable for all  $x\in L_p(M,\Phi)$.\\
%\item[(e)] Section $H(x)(\omega)\cdot H(y)(\omega)$ is measurable
%for all  $x,y\in M$.
%\end{enumerate}
% \end{theorem}

\section{Conditional expectations in $L_p(M,\Phi)$}

In this section we prove the existence of a conditional expectation
in noncommutative $L_p(M,\Phi)$-spaces associated with center-valued
traces, and give representation such operators as measurable bundle
of expectation operators in usual noncommutative $L_p$-spaces
associated with numerical traces.

Let $M$ be any finite von Neumann algebra, with unit element $\id$,
and $Z$ be some subalgebra of $Z(M)$. Then one can identify
 $Z=L_\infty(\Omega,\Sigma,m)$ and $S(Z) = L^0(\Omega,\Sigma,m)$.
 Let $\Phi$ be a strictly normal  trace on $M$ with values in  $L^0(\Omega)$, associated with  $Z$. Let
  $p\geq 1$ and $L_p(M,\Phi)$ be the noncommutative $L_p$-space, associated with von Neumann algebra  $M$ and the trace  $\Phi$.

\begin{proposition} \label{Exp1} Let $M$ and $\Phi$ be as above. Let $N$ be a von Neumann subalgebra of $M$, and the restriction $\Phi$ to $N$
also denoted by $\Phi$. Then there exists a unique map $\mathcal{E}:
M\rightarrow N$
 satisfying the following conditions:
 \begin{enumerate}
\item[(i)] $\mathcal{E}$ is a normal positive contraction from $M$
onto $N$;

\item[(ii)] $\mathcal{E}(axb)=a\mathcal{E}(x)b$ for any $x\in M$ and $a,b\in N;$

\item[(iii)] $\Phi(\mathcal{E}(x))=\Phi(x)$ for any $x\in M$.
\end{enumerate}
\end{proposition}

\begin{proof}  Without loss of generality, we may assume $\Phi(x)\in L^\infty(\Omega)$ (otherwise, we put $\Phi_0(x) =
 \Phi(x)(\id + \Phi(\id))^{-1}$.) Let $\nu$ be a faithful normal trace on $L^\infty(\Omega)$. Then $\tau(x)=\nu(\Phi(x))$
 is also a faithful normal numerical  trace on $M$. Hence, we can define a faithful normal numerical  trace $\tau$ on $M$. Then the existence of $\mathcal{E}$
 and properties (i) and (ii) follow from \cite[Proposition 2.1]{P}.

 Let us show (iii). Let $x\in M$. Then by Proposition 2.1. (iii) \cite{P} we
  get $\tau(\mathcal{E}(x))=\tau(x).$ By definition of $\tau$ one has
  $\nu(\Phi((\mathcal{E}(x))=\nu(\Phi(x))$ for any $x\in M$.

  Let
  $\Phi((\mathcal{E}(x))-\Phi(x)=v|\Phi((\mathcal{E}(x))-\Phi(x)|$
  be a polar decomposition of $\Phi((\mathcal{E}(x))-\Phi(x)$,
  for some partial isometry $v\in L^\infty(\Omega)$. Then one finds
  $$0=\nu(\Phi((\mathcal{E}(vx))-\nu(\Phi(vx))=\nu(v(\Phi((\mathcal{E}(x))-\Phi(x))=\nu(|\Phi(\mathcal{E}(x))-\Phi(x)|).$$

The faithfulness of $\nu$ implies that
$\Phi(\mathcal{E}(x))=\Phi(x)$ for any $x\in M$. Similarly, we can
show $\Phi(\mathcal{E}(x)y)=\Phi(xy)$ for any $x\in M, y\in N$.
 \end{proof}

\begin{theorem}\label{Exp2} The operator $\mathcal{E}$ (from Proposition \ref{Exp1}) can be extended to a positive $L_p(M,\Phi)$--contraction
from $L_p(M,\Phi)$ onto $L_p(N,\Phi)$, $1\leq p<\infty$.
Moreover, if $a,b\in N$ and $x\in L_p(M,\Phi)$ then
$$\mathcal{E}(axb)=a\mathcal{E}(x)b.$$
\end{theorem}

\begin{proof} Let $x\in M, y\in N$. Then $|\Phi(\mathcal{\mathcal{E}}(x)y)|=|\Phi(xy)|\leq \Phi(|xy|)
\leq \|y\|_N\|x\|_1$. Hence, according Proposition \ref{1.1} (i) we
have
$$\|\mathcal{E}(x)\|_1=\sup\{|\Phi(\mathcal{E}(x)y)|: \|y\|_M\leq
1\}\leq \|x\|_1.$$

 Let $p>1$. By Proposition \ref{1.1} (ii) one gets
$$|\Phi(\mathcal{\mathcal{E}}(x)y)|=|\Phi(xy)|\leq \Phi(|xy|)\leq \|y\|_q\|x\|_p,$$
where $\frac{1}{p}+\frac{1}{q}=1.$ Then one finds
$$\|\mathcal{E}(x)\|_p=\sup\{|\Phi(\mathcal{E}(x)y)|: \|y\|_q\leq
\mathbf{1}\}\leq\|x\|_p.$$

Hence  $\|\mathcal{E}(x)\|_p\leq\|x\|_p$, for any $x\in M$, i.e.
$\mathcal{E}$ is an $L_p(M,\Phi)$--contraction. Since $M$ is
$(bo)$-dense in $L_p(M,\Phi)$, so we can continue $\mathcal{E}$
until $L_p(M,\Phi)$--contraction to $L_p(M,\Phi)$, which is denoted
by the same symbol $\mathcal{E}$ (i.e.  $\mathcal{E}(x)=(bo) -
\lim\limits_{n\rightarrow\infty}\mathcal{E}(x_n)$, if $ x_n\in M$
 and $\|x_n-x\|_p\stackrel{(o)}\rightarrow0.$) It is clear that
 $\mathcal{E}(\id)=\id$ and $\|\mathcal{E}\|=\mathbf{1}$.

Let us show $\mathcal{E}$ is positive. Let $0\leq x\in L_p(M,\Phi)$.
We will chouse $0\leq x_n\in M$ such that $x_n$ is increasing and
$\|x_n-x\|_p\stackrel{(o)}\rightarrow0.$ Then the sequence
$\{\mathcal{E}(x_n)\}$ is increasing as well, and one has
$\mathcal{E}(x)=(bo) -
\lim\limits_{n\rightarrow\infty}\mathcal{E}(x_n)$. Since
$\mathcal{E}(x_n)\geq0$ and thw cone $K$ of positive elements of
$L_p(M,\Phi)$ is monotone closed (see\cite{GaC}), we get
$\mathcal{E}(x)\geq 0.$

Let $y\in L_p(N,\Phi)$ and $ y_n\in N$
 and $\|y_n-y\|_p\stackrel{(o)}\rightarrow0.$ Then
 $\mathcal{E}(y)=(bo) -
\lim\limits_{n\rightarrow\infty}\mathcal{E}(y_n)=(bo) -
\lim\limits_{n\rightarrow\infty}y_n=y$, i.e. $\mathcal{E}$ is
onto.

Now let $a,b\in N$,  $x\in L_p(M,\Phi)$ and $x_n\in M$ which
$\|x_n-x\|_p\stackrel{(o)}\rightarrow0.$ Then
\begin{eqnarray*}
\mathcal{E}(axb)&=&(bo) -
\lim\limits_{n\rightarrow\infty}\mathcal{E}(ax_nb)=(bo) -
\lim\limits_{n\rightarrow\infty}a\mathcal{E}(x_n)b\\[2mm]
&=&a((bo) -
\lim\limits_{n\rightarrow\infty}\mathcal{E}(x_n))b=a\mathcal{E}(x)b.
\end{eqnarray*}
 \end{proof}

\begin{definition}
The defined operator $\mathcal{E}: L_p(M,\Phi)\rightarrow
L_p(N,\Phi)$ in Theorem \ref{Exp2} is called a \textit{conditional
expectation}.
\end{definition}

 \begin{theorem}\label{2.3} Let $\mathcal{E}: L_p(M,\Phi)\rightarrow L_p(N,\Phi)$
 be the conditional expectation.
 Then there exists a measurable bundle of conditional expectations $\mathcal{E}_\omega: L_p(M(\omega),\tau_\omega)\rightarrow L_p(N(\omega),\tau_\omega)$
 such that
 $$\mathcal{E}(x)(\omega)= \mathcal{E}_\omega(x(\omega))$$ for all
$x\in L_p(M,\Phi)$ and for almost all $\omega\in\Omega$.
  \end{theorem}

   \begin{proof}
Let $\ell: M\rightarrow \mathcal{L}^\infty(\Omega,X)$ be the
noncommutative vector-valued lifting associated with lifting $\rho$.
 Since $\mathcal{E}(M)\subset N$ can define a linear operator
 $\varphi_\omega: \{\ell(x)(\omega): x\in M\}\rightarrow
 L_p(N(\omega),\tau_\omega)$ by
 $$\varphi_\omega(\ell(x)(\omega))=\ell(\mathcal{E}(x))(\omega).$$
 From the relations
\begin{eqnarray*}
\|\varphi_\omega(\ell(x)(\omega))\|_{L_p(N(\omega),\tau_\omega)}&=&
 \|\ell(\mathcal{E}(x))(\omega)\|_{L_p(N(\omega),\tau_\omega)}=\rho(\|\mathcal{E}(x)\|_{p})(\omega)\\[2mm]
 &\leq&\rho(\|x\|_p)(\omega)=\|\ell(x)(\omega)\|_{L_p(M(\omega),\tau_\omega)}
 \end{eqnarray*}
 follow that $\varphi_\omega$ is a well-defined contraction.

 Let $\ell(x)(\omega)\geq0$. From the equality
 $|\ell(x)|=\ell(|x|)$ we get
 $\ell(x)(\omega)=\ell(|x|)(\omega)$, which yields
 $$\varphi_\omega(\ell(x)(\omega))=\varphi_\omega(\ell(|x|)(\omega))=\ell(\mathcal{E}(|x|))(\omega)\geq0,$$
  i.e. $\varphi_\omega$ is positive.

 The density of $\{\ell(x)(\omega): x\in M\}$ in
  $L_p(M(\omega),\tau_\omega)$ implies that we can continue $\varphi_\omega$
  until a linear positive contraction $T_\omega: L_p(M(\omega),\tau_\omega)\rightarrow
  L_p(N(\omega),\tau_\omega)$ by
  $$T_\omega(x(\omega))=\lim\limits_{n\rightarrow\infty}\varphi_\omega(\ell(x_n)(\omega))=\lim\limits_{n\rightarrow\infty}\ell(\mathcal{E}(x_n))(\omega).$$

  Let $(X,L)$ be a measurable bundle of the noncommutative
  $L_p(M(\omega),\tau_\omega)$-spaces and $(Y,L)$ be a measurable bundle of the noncommutative
  $L_p(N(\omega),\tau_\omega)$-spaces.
Since $\varphi_\omega(\ell(x)(\omega))\in
\mathcal{L}^\infty(\Omega,Y)$, for any $x\in M$, we obtain that
$T_\omega(x(\omega))\in M(\Omega, Y)$ for any $x\in M(\Omega, X)$.
Therefore, $\{T_\omega\}$ is a measurable bundle of positive
contractions.

It is clear that
$\|T_\omega\|_{L_p(M(\omega),\tau_\omega)\rightarrow
L_p(N(\omega),\tau_\omega)}=\|\mathcal{E}\|(\omega)$. From
$\|\mathcal{E}\|=\mathbf{1}$, one has
$\|T_\omega\|_{L_p(M(\omega),\tau_\omega)\rightarrow
L_p(N(\omega),\tau_\omega)}=1$.

Let us show that $T_\omega$ is a conditional expectation from
$M(\omega)$ onto $N(\omega)$.

(i) Let $x(\omega)\in M(\omega)_{sa}(=\{x(\omega)\in M(\omega):
x(\omega)^*=x(\omega)\}).$ Then $|x(\omega)|\leq
\|x(\omega)\|_{M(\omega)}\id(\omega)$ and
\begin{eqnarray*}
|T_\omega(x(\omega))|&\leq&
T_\omega(|x(\omega)|)\leq\|x(\omega)\|_{M(\omega)}T_\omega(\id(\omega))\\[2mm]
&=&\|x(\omega)\|_{M(\omega)}\id(\omega),
\end{eqnarray*}
which means  $T_\omega(M(\omega))\subset N(\omega)$ and $T_\omega$
is  contraction from $M(\omega)$ to $N(\omega)$, here
$\id(\omega)=\ell(\id)(\omega)$.

Let $y(\omega)\in N(\omega).$ Then there exists $y_n\in N$ such that
$y(\omega)=\lim\limits_{n\rightarrow\infty}\ell(y_n)(\omega).$ Due
to  $\mathcal{E}(y_n)=y_n$ we get
\begin{eqnarray*}
T_\omega(y(\omega))&=&\lim\limits_{n\rightarrow\infty}T_\omega(\ell(y_n)(\omega))=
\lim\limits_{n\rightarrow\infty}\ell(\mathcal{E}(y_n))(\omega)\\[2mm]
&=&\lim\limits_{n\rightarrow\infty}\ell(y_n)(\omega))=y(\omega).
\end{eqnarray*}
This means $T_\omega(M(\omega))=N(\omega)$, i.e. $T_\omega$ maps
$M(\omega)$ onto $N(\omega).$

(ii) Let $a(\omega), b(\omega)\in N(\omega)$ and $x(\omega)\in
M(\omega)$. Then
$$T_\omega
(a(\omega)x(\omega)b(\omega))=\lim\limits_{n\rightarrow\infty}\ell(\mathcal{E}(a_nx_nb_n))(\omega)$$
for some $a_n, b_n\in N$ and $x_n\in M$. Due to
$\mathcal{E}(a_nx_nb_n))=a_n\mathcal{E}(x_n)b_n$ we find
\begin{eqnarray*}
T_\omega
(a(\omega)x(\omega)b(\omega))&=&\lim\limits_{n\rightarrow\infty}\ell(a_n)(\omega)\ell(\mathcal{E}(x_n))(\omega)\ell(b_n)(\omega)\\[2mm]
&=&a(\omega)T_\omega (x(\omega))b(\omega).
\end{eqnarray*}

(iii) Let $x(\omega)\in M(\omega)$. Then
$$\tau_\omega(T_\omega
(x(\omega))=\tau_\omega(\lim\limits_{n\rightarrow\infty}\ell(\mathcal{E}(x_n)(\omega)=\lim\limits_{n\rightarrow\infty}p(\Phi(\mathcal{E}(x_n))(\omega)$$
for some $x_n\in M$. From $\Phi(\mathcal{E}(x_n))=\Phi(x_n)$ one
gets $$\tau_\omega(T_\omega
(x(\omega))=\lim\limits_{n\rightarrow\infty}p(\Phi(x_n)(\omega)=\lim\limits_{n\rightarrow\infty}
\tau_\omega(\ell(x_n)(\omega))=\tau_\omega(x(\omega)).$$

Hence, $T_\omega$ is a conditional expectation from $M(\omega)$ onto
$N(\omega)$. So, it can be extended to $T_\omega:
L_p(M(\omega),\tau_\omega)\rightarrow L_p(N(\omega),\tau_\omega)$,
which is a conditional expectation from $L_p(M(\omega),\tau_\omega)$
onto $L_p(N(\omega),\tau_\omega)$, i.e.
$T_\omega=\mathcal{E}_\omega.$

 We show that $\mathcal{E}(x)(\omega)= \mathcal{E}_\omega(x(\omega))$ for all
$x\in L_p(M,\Phi)$ and for almost all $\omega\in\Omega$. It is
clear that  $\mathcal{E}_\omega x(\omega) =
 \mathcal{E}(x)(\omega)$ for $x\in M$ and all  $\omega\in\Omega$.

 Let $x\in L_p(M,\Phi)$. As
 $M$ $(bo)$-dense in $L_p(M,\Phi)$, then there exists a sequence
 $\{x_n\}\subset M$ such that $\|x_n-x\|_p\stackrel{(o)}\rightarrow0.$ Then  $\|x_n(\omega) -
 x(\omega)\|_{L_p(M(\omega),\tau_\omega)} \rightarrow0$ for almost all $\omega\in \Omega$.
 From the equality  $\mathcal{E}(x)=(bo) -
\lim\limits_{n\rightarrow\infty}\mathcal{E}(x_n)$ we get
$$\|\mathcal{E}_\omega(x_n(\omega)) -
 (\mathcal{E}x)(\omega)\|_{L_p(M(\omega),\tau_\omega)} \rightarrow0$$ for almost all $\omega\in \Omega$, i.e.
 $\lim\limits_{n\rightarrow\infty}\mathcal{E}_\omega(x_n(\omega))
 =
 (\mathcal{E}x)(\omega)$ for almost all $\omega\in \Omega$.

 On the other hand, from the continuity of $\mathcal{E}_\omega$
we obtain $\lim\limits_n \mathcal{E}_\omega (x_n(\omega)) =
\mathcal{E}_\omega
 x(\omega)$ a.e. Hence $\mathcal{E}(x)(\omega)= \mathcal{E}_\omega(x(\omega))$ for all
$x\in L_p(M,\Phi)$ and for almost all $\omega\in\Omega$. This
completes the proof.
\end{proof}

\section{Martingales theorems in noncommutative $L_p(M,\Phi)$ spaces}

In this section, as an application of the results of the previous
section, we are going to prove the norm-convergence of martingales
in the noncommutative $L_p(M,\Phi)$ spaces.

\begin{definition}\label{3.1}\cite{P} Let $\{M_n\}_{n\geq1}$ be a sequence of von Neumann subalgebras of $M$ such that $M_1\subset M_2\subset
\dots$. If one has $\overline{\bigcup\limits_{n}M_n}^{w^*}=M$ then
the sequence $\{M_n\}_{n\geq1}$ is called a \textit{filtration} of
$M$.
\end{definition}

Let $M$ be as in Section 4, and $\{M_n\}$ be a filtration of $M$.
Let $\mathcal{E}(\cdot|M_n)$ be a conditional expectation from
$L_p(M,\Phi)$ to $L_p(M_n,\Phi)$.

\begin{definition}\label{3.2}A sequence $\{x_n\}\subset
L_p(M,\Phi)$ is called an \textit{$L_p(M,\Phi)$-martingale} with
respect to $\{M_n\}_{n\geq1}$ if one has
$\mathcal{E}(x_{n+1}|M_n)=x_n$ for every $n\geq 1$.
\end{definition}

\begin{proposition}\label{3.3} If  $\{x_n\}\subset
L_1(M,\Phi)$ is  an $L_1(M,\Phi)$-martingale with respect to
$\{M_n\}_{n\geq1}$ then $\{x_n(\omega)\}\subset
L_1(M(\omega),\tau_\omega)$ is also martingale with respect to
filtration $\{M_n(\omega)\}_{n\geq1}$ of
$M_\infty(\omega):=\overline{\bigcup\limits_{n}M_n(\omega)}^{w^*}$
for almost all $\omega\in \Omega$.
\end{proposition}

\begin{proof} Let $\mathcal{E}(x_{n+1}|M_n)=x_n$ for every
$n\geq1$. According Theorem \ref{2.3} we have
$$\mathcal{E}_\omega(x_{n+1}(\omega)|M_n(\omega))=\mathcal{E}(x_{n+1}|M_n)(\omega)=x_n(\omega)$$
for almost all $\omega\in \Omega$, i.e. $\{x_n(\omega)\}\subset
L_1(M(\omega),\tau_\omega)$ is a martingale with respect to the
filtration $(M_n(\omega))_{n\geq1}$ of
$\overline{\bigcup\limits_{n}M_n(\omega)}^{w^*}$.

\end{proof}

If for a  \textit{$L_p(M,\Phi)$-martingale} $\{x_n\}$ one has
$\sup\limits_{n\geq 1}\|x_n\|_p \in L^0(\Omega)$ then $\{x_n\}$ is
called a \textit{bounded $L_p(M,\Phi)$-martingale}.

\begin{theorem}\label{3.5}
Let $p>1$. The following statements hold:
\begin{enumerate}
\item[(i)] Let $x\in L_p(M,\Phi)$, $p\geq1$ and define
$x_n=\mathcal{E}(x|M_n)$. Then $\{x_n\}$ is a bounded
$L_p(M,\Phi)$-martingale, and  $x_n\rightarrow x$ in $L_p(M,\Phi)$.

\item[(ii)] Let $\{x_n\}$ be a bounded
$L_p(M,\Phi)$-martingale. Then there exists $x\in L_p(M,\Phi)$ such
that $x_n=\mathcal{E}(x|M_n)$.
\end{enumerate}

\end{theorem}

\begin{proof} (i). Due to  $\|x_n\|_p=\|\mathcal{E}(x|M_n)\|_p\leq
\|x\|_p$ we get
$$\sup\limits_{n\geq 1}\|x_n\|_p\leq \|x\|_p\in
L^0(\Omega).$$
 According Theorem \ref{2.3} one has
$$x_n(\omega)=\mathcal{E}(x|M_n)(\omega)=\mathcal{E}_\omega(x(\omega)|M_n(\omega))$$
for almost all $\omega\in \Omega$. Then using Proposition 2.10 (i)
\cite{P}  we obtain
$$\|x_n(\omega)-x(\omega)\|_{L_p(M(\omega),\tau_\omega)}\rightarrow
 0$$ for almost all $\omega\in \Omega$. From
 $$\|x_n-x\|_p(\omega)=\|x_n(\omega)-x(\omega)\|_{L_p(M(\omega),\tau_\omega)}$$
 it follows that $\|x_n-x\|_p\stackrel{(o)}\rightarrow0$ i.e. $x_n\rightarrow x$ in $L_p(M,\Phi)$.

 (ii). Let $x_n\in L_p(M,\Phi)$ be an $L_p(M,\Phi)$-martingale, i.e. $a:=\sup\limits_{n\geq 1}\|x_n\|_p \in
L^0(\Omega)$. Then $$\sup\limits_{n\geq
1}\|x_n(\omega)\|_{L_p(M(\omega),\tau_\omega)} =\sup\limits_{n\geq
1}\|x_n\|_p(\omega)=a(\omega)<\infty$$ for almost all $\omega\in
\Omega$. According Proposition 2.10 (ii) \cite{P}  there exists
$x(\omega)\in L_p(M(\omega),\tau_\omega)$, such that
$x_n(\omega)=\mathcal{E}_\omega(x(\omega)|M_n(\omega))$ for almost
all $\omega\in \Omega$. Again by Proposition 2.10 (i) \cite{P} we
have
$\|x_n(\omega)-x(\omega)\|_{L_p(M(\omega),\tau_\omega)}\rightarrow0$
for almost all $\omega\in \Omega$. This mean $x\in M(\Omega, X)$.
Therefore, $x\in L_p(M,\Phi)$. So
$$x_n=\widehat{\mathcal{E}_\omega(x(\omega)|M_n(\omega))}=\mathcal{E}(x|M_n).$$
This completes the proof.
 \end{proof}

 \begin{corollary}\label{3.6} Let $p>1$. The following statements hold:
 \begin{enumerate}
\item[(i)] Let $x_n\rightarrow x$ in $L_p(M,\Phi)$, then $\mathcal{E}(x_n|N)\rightarrow
\mathcal{E}(x|N)$ in $L_p(M,\Phi)$ for any von Neumann subalgebra
$N$ of $M$.

\item[(ii)] If $x_n\rightarrow x$ in $L_p(M,\Phi)$ and $\overline{\bigcup\limits_{m}M_m}^{w^*}=M$, then $\mathcal{E}(x_n|M_m)\rightarrow x$ in $L_p(M,\Phi)$
as $n,m\rightarrow\infty$.
\end{enumerate}

\end{corollary}

\begin{proof} (i) follows from Theorem \ref{2.3}.

(ii). It is clear that
\begin{eqnarray}\label{36-1}\|\mathcal{E}(x_n|M_m)-x\|_p&=&\|\mathcal{E}(x_n|M_m)-\mathcal{E}(x|M_m)+\mathcal{E}(x|M_m)-x\|_p\nonumber\\[2mm]
&\leq& \|x_n-x\|_p+\|\mathcal{E}(x|M_m)- x\|_p.
\end{eqnarray}
By condition we have $\|x_n-x\|_p\stackrel{(o)}\rightarrow0$, and
Theorem \ref{3.5} (i) implies $\|\mathcal{E}(x|M_m)-
x\|_p\stackrel{(o)}\rightarrow0$. Hence, from \eqref{36-1} one finds
$\mathcal{E}(x_n|M_m)\rightarrow x$ in $L_p(M,\Phi)$ as
$n,m\rightarrow\infty$.
\end{proof}

Now we are going to prove the convergence of weighted averages of
martingales in noncommutative $L_p(M,\Phi)$ spaces.

Let $\mathbf{x}=\{x_n\}$ be an $L_p(M,\Phi)$-martingale and
$\{w_n\}_{n\geq1}$ be a sequence
 of positive numbers such that
$W_n=\sum\limits_{k=1}^nw_k\rightarrow\infty$ as
$n\rightarrow\infty$. \textit{The weighted average} (see
\cite{ki1},\cite{ZH})\ $\sigma_n(\mathbf{x})$ of the martingale
$x=\{x_n\}$ is given by
$$\sigma_n(\mathbf{x})=\frac{1}{W_n}\sum\limits_{k=1}^nw_k x_k.$$

\begin{proposition} \label{3.8}  Let $\mathbf{x}=\{x_n\}$ be an $L_p(M,\Phi)$-martingale and
$\sigma_n(\mathbf{x})$ be a weighted average of $\mathbf{x}$. Then
one has
$$\sup\limits_{n\geq 1}\|x_n\|_p=\sup\limits_{n\geq 1}\|\sigma_n(\mathbf{x})\|_p.$$
\end{proposition}

\begin{proof} Due to $$(\sup\limits_{n\geq 1}\|x_n\|_p)(\omega)=\sup\limits_{n\geq
1}\|x_n(\omega)\|_{L_p(M(\omega),\tau_\omega)}$$ for almost all
$\omega\in \Omega$, from Theorem 3.2 \cite{ZH} we obtain
\begin{eqnarray*}
(\sup\limits_{n\geq 1}\|x_n\|_p)(\omega)&=&\sup\limits_{n\geq
1}\|x_n(\omega)\|_{L_p(M(\omega),\tau_\omega)}\\[2mm]
&=&\sup\limits_{n\geq
1}\|\sigma_n(\mathbf{x}(\omega))\|_{L_p(M(\omega),\tau_\omega)}\\[2mm]
&=&\sup\limits_{n\geq 1}\|\sigma_n(\mathbf{x})\|_p(\omega)
\end{eqnarray*}
for almost all $\omega\in \Omega$, here we have used the notation
$\mathbf{x}(\omega))=\{x_n(\w)\}$.  Hence, we obtain the required
equality.
\end{proof}

\begin{theorem} \label{3.9}  Let $\mathbf{x}=\{x_n\}$ be an $L_p(M,\Phi)$-martingale and
$\sigma_n(x)$ be a weighted average  of $\mathbf{x}$.  Then one has
$\{x_n\}_{n\geq1}$
 converges in $L_p(M,\Phi)$ if and only if
 $\sigma_n(\mathbf{x})$  converges in $L_p(M,\Phi)$.

\end{theorem}

\begin{proof} Let $x_n\rightarrow y$ in  $L_p(M,\Phi)$. Then $x_n(\omega)\rightarrow y(\omega)$ in $L_p(M(\omega),\tau_\omega)$ for almost all $\omega\in
\Omega$. According to Theorem 3.7 \cite{ZH}
$\sigma_n(\mathbf{x}(\omega))\rightarrow y(\omega)$ for almost all
$\omega\in \Omega$. Then
$\sigma_n(\mathbf{x})(\omega)=\sigma_n(\mathbf{x}(\omega))\rightarrow
y(\omega)$ in $L_p(M(\omega),\tau_\omega)$ for almost all $\omega\in
\Omega$. Hence
$$\|\sigma_n(\mathbf{x})-y\|_p(\omega)=\|\sigma_n(\mathbf{x}(\omega))-y(\omega)\|_{L_p(M(\omega),\tau_\omega)}\rightarrow
0$$ for almost all $\omega\in \Omega$, which means
$\sigma_n(\mathbf{x})$ converges in $L_p(M,\Phi)$.

Let $\sigma_n(\mathbf{x})\rightarrow y$ in $L_p(M,\Phi)$. Then
$\sigma_n(\mathbf{x}(\omega))\rightarrow y(\omega)$ in
$L_p(M(\omega),\tau_\omega)$ for almost all $\omega\in \Omega$.
Again by Theorem 3.7 \cite{ZH} $x_n(\omega)\rightarrow y(\omega)$ in
$L_p(M(\omega),\tau_\omega)$ for almost all $\omega\in \Omega$.
Hence $\|x_n-y\|_p\stackrel{(o)}\rightarrow0.$ This completes the
proof.
\end{proof}

\section*{Acknowledgments}
 The first author (I.G) acknowledges the  MOHE Grant FRGS13-071-0312.

\bibliographystyle{amsplain}

\begin{thebibliography}{99}

\bibitem{AAK} Albeverio S,  Ayupov Sh A., Kudaybergenov K K,
 Non commutative Arens algebras and their derivations.
J. Funct. Anal. 2007; 253: 287--302.

\bibitem{CP} C. Cecchini, D. Petz, Norm convergence of generalized martingales in $L^p$-spaces over von Neumann
algebras, \textit{Acta Sci. Math. (Szeged)} \textbf{48} (1985),
55-63.

\bibitem{CGa} V.I. Chilin. I.G. Ganiev,
An individual ergodic theorem for contractions in the
Banach--Kantorovich lattice $L^p(\nabla,\mu)$. \textit{Russian
Math.} {\bf 44}(2000), 77--79.

\bibitem{CK} V.I. Chilin, A.A.  Katz, On abstract characterization of non-commutative $L_p$-spaces
associated with center- valued trace, \textit{Methods Funct. Anal.
Topol.} {\bf 11}(2005), 346-–355

\bibitem{CZ} V.I. Chilin, B.S. Zakirov, Non-commutative $L_p$-spaces associated with a Maharam
trace,\textit{J. Operator Theory} {\bf 68} (2012), 67--83.

\bibitem{CZ1} V.I. Chilin, B.S. Zakirov, Maharam traces on von Neumann algebras,
\textit{Methods Funct. Anal. Topol.} {\bf 16}(2)(2010), 101--111.


\bibitem{DN} N. Dang-Ngoc, Pointwise convergence of martingales in von Neumann algebras, \textit{Israel J. Math.} {\bf 34} (1979), 273--280.

\bibitem{Dik} J. Dixmier, \textit{Les
$C^*$- -algebres et leurs representations}, Paris, Gauthier -
Villars Editeur, 1969.

\bibitem{Ga} I.G. Ganiev, The martingales  convergence in the
Banach-Kantorovich's  lattices
$L_p(\widehat{\nabla},\widehat{\mu}),$  \textit{Uzb. Math. Jour.}
2000, N.1, 18--26.

%\bibitem{Ga1} I.G. Ganiev, V.I. Chilin, Measurable bundles of non-commutative
%$L_p$-spaces associated with center-valued trace, \textit{Mat.
%Trudy}, {\bf 4}(2)(2001), 27–41. (Russian).


\bibitem{GaC} I.G. Ganiev, V.I. Chilin, Measurable bundles of
noncommutative $L^p$-spaces associated with a center-valued trace,
\textit{Siberian Adv. Math.} {\bf 12} (2002), 19--33.

\bibitem{GC1}  I.G. Ganiev, V.I. Chilin, Measurable bundles of $C^*$-algebras,
\textit{Vladikavkaz. Mat. Zh.} {\bf 5}(2003), N. 1, 35--38
(Russian).


\bibitem{GM} I.G. Ganiev, F. Mukhamedov,  On the "Zero-Two" law for positive
contractions in the Banach--Kantorovich lattice $L^p(\nabla,\mu)$,
\textit{ Comment. Math. Univ. Carolinae} {\bf 47}(2006), 427--436.

\bibitem{GM2} I. Ganiev, F. Mukhamedov,  On weighted ergodic theorems for
Banach--Kantorovich lattice $L_p(\nabla,\mu)$, \textit{
Lobachevskii Jour. Math.} {\bf 34}(2013), 1--10.

\bibitem{GM3}  I. Ganiev, F. Mukhamedov,  Measurable bundles of
$C^*$-dynamical systems and its applications, \textit{Positivity}
{\bf 18}(2014), 687--702.

\bibitem{GM4}  I. Ganiev, F. Mukhamedov, Weighted ergodic theorem for
contractions of Orlicz-Kantorovich lattice
$L_{M}(\widehat{\nabla},\widehat{\mu})$, \textit{Bull. Malay.
Math. Sci. Soc.} {\bf 38}(2015), 387--397.

\bibitem{G} S. Goldstein, Norm convergence of martingales in $L^p$-spaces over von Neumann
algebras, \textit{Rev. Roumanie Math. Pures Appl.} \textbf{32}
(1987), 531-541.

\bibitem{G0} A.E. Gutman,  Banach bundles in the theory of
 lattice--normed spaces. I. Continuous Banach bundles, \textit{Siberian
 Adv. Math.} {\bf 3}(1993),  1--55

\bibitem{G1} A.E. Gutman,  Banach bundles in the theory of
lattice-normed spaces, III, \textit{Siberien Adv. Math.} {\bf
3}(1993), 8--40.

\bibitem{G2} A.E. Gutman, Banach fiberings in the theory of lattice-normed spaces.
Order-compatible linear operators. \textit{Trudy Inst. Mat.} {\bf
29} (1995), 63--211. (Russian)

\bibitem{HT} F. Hiai, M.Tsukada, Generalized conditional expectations and martingales in noncommutative $L^p$-spaces,
\textit{J. Operator Theory} \textbf{18} (1987), 265-288.

\bibitem{ki1} M. Kikuchi, Convergence of weighted averages of martingales
in Banach function spaces, \textit{J. Math. Anal. Appl.} {\bf 244}
(2000), 39--56.

\bibitem{Kus} A.G. Kusraev,   \textit{Vector duality and its
applications}, Nauka, Novosibirsk, 1985.  (in Russian).


\bibitem{K1} A.G. Kusraev,  \textit{Linear operators in lattice-normed spaces,} in: Studies on Geometry in the Large and Mathematical Analysis.
 Vol. 9, Nauka, Novosibirsk,  1987.  (in Russian).

\bibitem{K2} A.G. Kusraev,
      \textit{ Dominated operators,} Kluwer Academic Publishers, Dordrecht,  2000.

\bibitem{La} E. C. Lance, Martingale convergence in Neumann algebras, \textit{Math. Proc. Cambridge Philos. Soc.}  {\bf 84} (1978), 47--56.

\bibitem{Mur_m}
M.A. Muratov, V.I. Chilin, \textit{  Algebras of measurable and
locally measurable operators}, ~Kyiv, Pratsi In-ty matematiki NAN
Ukraini, 2007~(Russian).


\bibitem{N} E. Nelson, Note on non-commutative integration,
\textit{J. Funct. Anal.} \textbf{15} (1974), 103--117.

\bibitem{S} I.E. Segal, A non-commutative extansion of abstract integration,
\textit{Ann. Math.} \textbf{57} (1953), 401--457.


\bibitem{Ta} M. Takesaki, Conditional expectations in von Neumann algebras, \textit{J. Funct. Anal.} {\bf 9} (1972), 306--321.

\bibitem{Ts} M. Tsukada, Strong convergence of martingales in von Neumann algebras, \textit{Proc. Amer. Math. Soc.}
{\bf 88} (1983), 537--540.

\bibitem{U1} H. Umegaki, Conditional expectation in an operator
algebra, \textit{Tohoku Math. J.} {\bf 6}(1954), 177-181.

\bibitem{U2} H. Umegaki, Conditional expectation in an operator
algebra II, \textit{Tohoku Math. J.} {\bf 8}(1956), 86-100.

\bibitem{P} Q. Xu,  Operator spaces and noncommutative $L_p$-
spaces. The part on noncommutative $L_p$-spaces. Lectures in the
Summer School on Banach spaces and Operator spaces Nankai
University China July 16 - July 20, 2007.

\bibitem{Y} F.J. Yeadon, Non-commutative $L_p$ -spaces, \textit{Math. Proc. Camb.
Phil. Soc.} {\bf 77} (1975), 91–-102.

\bibitem{ZH} Ch. Zhang, Y. Hou,  Convergence of weighed averages of martingales in noncommutative
Banach function spaces, \textit{Acta Math. Scientia} {\bf
32B}(2012), 735-–744.

\end{thebibliography}

\end{document}